\documentclass[11pt]{article}

\usepackage[
  textheight=190mm,
  left=17.5mm,
  right=17.5mm
]{geometry}

\usepackage{lipsum}
\usepackage{amsfonts,amsmath,amsthm}
\usepackage{graphicx}
\usepackage{epstopdf}
\usepackage{algorithmic}
\ifpdf
\DeclareGraphicsExtensions{.eps,.pdf,.png,.jpg}
\else
\DeclareGraphicsExtensions{.eps}
\fi

\usepackage{hyperref}

\newcommand{\e} {\varepsilon}
\newcommand{\p}{{\mathbf P}}
\newcommand{\R}{{\mathbb R}}
\newcommand{\C}{{\mathbb C}}

\newcommand{\A}{{\mathcal A}}

\newcommand{\N}{{\mathbb N}}

\newcommand{\eps} {\varepsilon}

\newcommand{\di}[1]{\operatorname{d}\!#1}
\newcommand{\im}{\operatorname{im}}

\newtheorem{theorem}{Theorem}[section]
\newtheorem{assumption}[theorem]{Assumption}
\newtheorem{proposition}[theorem]{Proposition}
\newtheorem{definition}[theorem]{Definition}
\newtheorem{remark}[theorem]{Remark}
\newtheorem{lemma}[theorem]{Lemma}
\newtheorem{corollary}[theorem]{Corollary}

\title{Ensemble Feedback Methods for Families of Linear Systems}

\author{Michael Sch\"onlein\thanks{Bauhaus-Universit\"at Weimar, Germany 
		(email: michael.schoenlein$\{AT\}$uni-weimar.de)}
	\and Fabian Wirth\thanks{University of Passau, Germany 
		(email: fabian.lastname$\{AT\}$uni-passau.de)}
	}

\usepackage{amsopn}

\begin{document}

\maketitle

\begin{abstract}
We consider feedback stabilization for one-parameter families of finite-dimensional linear systems over compact parameter sets in the complex field. Classical ensemble feedback induces compact control operators and therefore cannot modify the essential spectrum of the associated multiplication operators describing the free motion of the system. This precludes stabilization in many infinite-dimensional settings. To address this issue, multiplication feedback operators are introduced. For systems with constant Hermite indices, an analogue of Heymann's lemma is proved,  as well as a pole placement theorem, and stabilization results. The relation of pointwise and ensemble controllability is investigated. For systems with nonconstant Hermite indices, corresponding results are obtained under additional assumptions on the structure of the parameter set.
\end{abstract}



\section{Introduction}

Ensemble control is a rather new research area of control theory which is concerned with a parameter-dependent
``family of systems'' (i.e. an ensemble) instead of a single one. Here, the major challenge is to perform classical control
tasks simultaneously, i.e.~for the entire ensemble via controls that are \emph{independent} of the system parameter.
Starting with the contributions \cite{bcr2010,li2009ensemble}, the topic of ensemble controllability has become an active and growing area in mathematical control theory, cf. \cite{agrachev-baryshnikov-sarychev2016,agrachev2020control,chen2019mcss,Xudong_SICON_2023,Jerome_MCRF_2023,JDE_ensembles_2021,banff_2023,helmke2014uniform,LAZAR2022265,li2011,li_tac_2016,li2020separating,MCSS_Ensembles,schonlein2016controllability,Zeng_scl_2016,shen2017discrete,zeng2016ensemble,zhang2018controllability,zuazua2014averaged}. Ensemble control is located at the crossroad of finite- and infinite-dimensional control theory, operator theory and approximation theory.

The problem of simultaneous stabilization of a parameterized  family of linear systems falls into this setting. In this context, parameter-dependent pole-shifting has been addressed in the 1980s and 1990s. We note that these contributions used {rather} different methods. On the one hand, in \cite{hautussontag1986,sontag_intro_families,sontagwang1990} the algebraic theory of systems over rings is used.
On the other hand, a frequency domain approach for the simultaneous stabilization problem using function theoretic methods was proposed by various authors. For details and more references we refer to the comprehensive monograph \cite{blondel1994simultaneous}  and the more recent contributions \cite{Chen-feedback-Banff,guth2023ensemble, ryan2014simultaneous}. In this work, we will use tools from functional analysis and approximation theory to study the possibilities and limits of ensemble feedback methods for one-parameter families of linear systems.

The paper is organized as follows. In the remainder we introduce notation, recall basic concepts and discuss the novelty of this paper. In Section~\ref{sec:prelims} we recall some facts about matrix multiplication operators. In Proposition~\ref{prop:spec-mult-operator} it is shown that the spectrum of multiplication operators on the space of continuous functions on a perfect parameter space is essential;
a fact we were not able to locate in the literature. Some basics concerning the stability of semigroups generated by bounded operators are recalled. In Section~\ref{sec:staticfeedback} we start our investigation of the stabilization problem using static state feedback. We distinguish between ensemble feedback and multiplication feedback in which the feedback operator is allowed to be parameter dependent. If $\p$ is perfect, then ensemble feedback usually does not allow for stabilization as the essential spectrum is not changed by a compact perturbation. In the case of multiplication feedback we consider the case that the Hermite indices are constant over the parameter space. In this case we are able to give a proof of a version of Heymann's lemma that provides explicit formulas for the desired feedback. In this way the pole-placement and stabilization theorems can be obtained. An application of approximation theoretic results is used to show that if  Hermite indices are constant and certain assumptions on the parameter set that make approximation theoretic results applicable are satisfied, then a multiplication feedback can generate a uniformly ensemble reachable system from a pointwise reachable one. In Section~\ref{sec:homotopy} we briefly indicate an alternative route to parameter-dependent versions of Heymann's lemma that analyze what we call the Heymann bundle over the parameter set $\p$. In this language the veracity of a parametrized Heymann lemma translates into the existence of a global section of the respective bundle. For the simple cases that the parameter set is an interval or a circle a positive answer can be provided here without any assumption on the Hermite indices. We conclude in Section~\ref{sec:conclusions} and point to further open questions. In the Appendix, Section~\ref{sec:appendix}, we provide a characterization of ensemble reachability which summarizes our application of approximation theory. We also provide a short proof of a folklore result that is frequently used without proof. Namely that for a continuous square matrix function with continuous injective eigenfunctions and pairwise disjoint and simple spectra there exists a continuous diagonalization.

\subsection*{Setting and Notation}

The fields of real and complex numbers are denoted $\R,\C$; $\C^n$ and $\C^{n\times m}$ are the spaces of column vectors of size $n$, respectively $n\times m$ matrices with complex entries. For Banach spaces $X,Y$ we denote by $L(X,Y)$ the space of bounded linear operators from $X$ to $Y$ and $L(X):= L(X,X)$.

Let $\p \subset \C$ be the parameter space. We will assume throughout that $\p \subset \C$  is compact. 
Sometimes it will be necessary to assume that $\p$ is perfect, i.e., it does not contain isolated points in addition to being compact. For the approximation theoretic considerations in the appendix further properties are required. 
Let $X(\p)$ be an arbitrary separable Banach space of functions defined on $\p$ with values in $\C$ and let $X_{n,m}(\p)$ denote the space of $(n \times m)$-matrices with entries in $X(\p)$. Furthermore, set $X_n(\p) :=  X_{n,1}(\p)$. Thus $X_n(\p)$ is simply the $n$-fold Cartesian product of $X(\p)$ and therefore again a Banach space. Unless stated otherwise, we assume that $X_n(\p)$ is equipped with the maximum norm of the entrywise norms, i.e.~$\|x\|_{X_n(\p)} := \max_{1 \leq i \leq n}\|x_i\|_{X(\p)}$. As ususal, $C(\p)$ denotes the Banach space of continuous complex-valued functions on $\p$. By $C_n(\p)$ respectively $C_{n,m}(\p)$ we denote the spaces of $n$-vectors, $n\times m$-matrices with entries in $C(\p)$.

It is easy to see that for any $B \in X_{n,m}(\p)$ the input operator
\begin{equation}
\label{eq:def-input-operator}
    {\mathcal B}:\C^m \to X_n(\p), \quad ({\mathcal B}v)(\theta) := B(\theta)v, \quad v \in \C^m, \theta\in \p,
\end{equation}
is well-defined, linear and bounded as it is a multiplication operator with finite dimensional domain $\C^m$. A structural assumption on $X_n(\p)$ is that for any $A \in C_{n,n}(\p)$ the induced multiplication
operator
\begin{equation}
	\label{eq:def-operators-scalar2}
	{\mathcal A}: X_n(\p) \to X_n(\p),\quad ({\mathcal A}f)(\theta) := A(\theta)f(\theta), \quad \theta \in \p,
\end{equation}
is well-defined, linear and bounded.
We formalize this in the following assumption.
\begin{assumption}
\label{ass:basicP}
 The set $\p\subset \C$ is nonempty and compact. The space $X_n(\p)$ is a Banach space of $\C^n$ valued functions such that for every $A\in C_{n,n}(\p)$ the operator $\A$ defined in \eqref{eq:def-operators-scalar2} is a bounded linear operator on $X_n(\p)$.
\end{assumption}

\begin{remark}
 Recall that a Banach algebra $(B,\|\cdot\|)$ is a Banach space equipped with a bilinear multiplication $B\times B\to B$ that is associative and satisfies $\|xy\| \leq \|x\|\|y\|$ for all $x,y\in B$.  A Banach $B$-module is a Banach space $Z$ together with a continuous algebra homomorphism $\rho : B \to L(Z)$. 
 The space $C(\p)$ with the standard supremum norm is a Banach algebra together with pointwise multiplication of functions. 
To rephrase our structural assumption on $X_n(\p)$, we are requiring that $X_n(\p)$ is a Banach module over $C(\p)$. Standard spaces for which this holds are the spaces of continuous functions, $L^p$ spaces or subspaces of such spaces of functions vanishing on a fixed closed subset of $\p$. If $X_n(\p)$ contains all constant functions it is not hard to see that $C_n(\p) \subset X_n(\p)$ is required.
For examples not satisfying the assumption consider spaces of Lipschitz or Hölder functions, or assume that $\p$ is a subset of the unit disc and that $X_n(\p)$ is the space of restrictions of bounded holomorphic functions on the disc  to $\p$ (i.e. $H^\infty$ restricted to $\p$).   
\end{remark}

From now on let the parameter space $\p$ and a Banach space of functions $X_n(\p)$ be fixed so that our standing assumptions are satisfied. Consider $A \in C_{n,n}$ and $B\in X_{n,m}(\p)$ with associated bounded linear operators $\mathcal{A} \in L(X_n(\p))$, $\mathcal{B}\in L(\C^m, X_n(\p))$ defined in \eqref{eq:def-operators-scalar2} resp. \eqref{eq:def-input-operator}. The parameter-dependent system under consideration is of the form
\begin{equation}
	\begin{split}
		\label{eq:sys-par2}
		\tfrac{\partial}{\partial t} x(t,\theta)=A(\theta)x(t,\theta)+B(\theta)u(t) 
        , \quad \theta \in \p,
	\end{split}
\end{equation}
or in discrete-time
\begin{equation}\label{eq:sys-par-discrete2}
	x(t+1,\theta)=A(\theta)x(t,\theta)+B(\theta)u(t), 
    \quad \theta \in \p,
\end{equation}
where $x_0 \in X_n(\p)$, $x(0,\theta)=x_0(\theta)\in\C^n$ denotes the initial conditions for the system with parameter $\theta$. In terms of these matrix multiplication operators, the dynamic equations  \eqref{eq:sys-par2} and \eqref{eq:sys-par-discrete2}
are equivalent to the (infinite dimensional) linear control systems
\begin{align}
		\label{eq:sys-inf2}
		\dot x (t)&={\mathcal A}x(t) + {\mathcal B}u(t),\quad x(0) =x_0 \in X_n(\p),
	\intertext{and correspondingly}
	\label{eq:sys-inf-discrete2}
		x({t}+1)&={\mathcal A}x({t}) + {\mathcal B}u({t}),\quad x(0) =x_0 \in X_n(\p)\,.
		\end{align}
In the continuous-time case, the space of admissible (open-loop) input functions is $L^1_{\mathrm{loc}}(\R_+,\C^m)$ and in the discrete-time case it is $(\C^m)^\N$. From now on $T\geq 0$ is either a nonnegative real or a nonnegative integer and we abbreviate
$U(T):=L^1([0,T],\mathbb{C}^m)$ or
$U(T):=(\C^m)^{T+1}$ depending on the system under consideration. Solutions to
\eqref{eq:sys-par2} and \eqref{eq:sys-par-discrete2} are denoted $  \varphi \big (T,x_0,u\big ) (\theta) :=\varphi \big (T,x_0(\theta),u,\theta \big )$, 
where in the continuous-time case
\begin{equation*}
	\varphi \big (T,x_0,u \big ) (\theta ) =  {\mathrm{e}}^{TA(\theta )}  x_0(\theta) + \int _{0}^{T} {\mathrm{e}}^{(T-\tau )A(\theta )}B(\theta )u(
	\tau )\di{\tau}
\end{equation*}
and in discrete-time
\begin{equation*}
	\varphi \big (T,x_0,u \big ) (\theta ) =  A(\theta )^{T}x_0(\theta)+\sum_{k=0}^{T-1} A(\theta )^{T-1-k}B(\theta )u(k ).
\end{equation*}
In the following we will identify a pair $(A,B) \in C_{n,n}(\p) \times X_{n,m}(\p)$ with the system \eqref{eq:sys-inf2} or \eqref{eq:sys-inf-discrete2} and briefly speak of "the system $(A,B)$". A central notion for this paper is the following version of reachability. We refer to standard literature for the definition of the notion of reachability for finite-dimensional linear systems, e.g. \cite[Chapter~3]{sontag}, \cite[Section 6.1.4]{hinrichsen2026mathematicalII}.

\begin{definition}
\label{def:reachbility}
	A pair $(A,B) \in C_{n,n}(\p)\times X_{n,m}(\p)$ is called
	\begin{enumerate}
		\item[(i)] \textit{pointwise reachable at $\theta_0 \in \p$} if the pair $(A(\theta_0),B(\theta_0))
		\in \C^{n\times n}\times \C^{n\times m}$
		is reachable.
		\item[(ii)] pointwise reachable if it is pointwise reachable at every $\theta_0 \in \p$.
		\item[(iii)] \textit{ensemble reachable} with respect to $X_{n}(\p )$, if for all $f \in X_{n}(\p)$ and $\eps >0$ there exist $T \geq 0$ and an
		input $u \in {U(T)}$ such that
		\begin{equation*}
						\|\varphi (T,u,0) -f\|_{X_{n}(\p)} < \eps .
		\end{equation*}
		\item[(iv)] \textit{uniformly ensemble reachable} if it is ensemble reachable with respect to $C_n(\p)$.
	\end{enumerate}
\end{definition}

\medskip
We note that for $X_n(\p) =C_n(\p)$, we have that uniform ensemble reachability implies pointwise reachability, cf. \cite[Lemma~1]{helmke2014uniform}. In the previous reference $\p$ is assumed to be an interval. However, inspection of the proof shows that the argument extends readily to compact $\p$. In the case $X_n(\p) = L_n^q(\p)$ ensemble reachability (w.r.t. to  $L_n^q(\p)$) implies pointwise reachability for almost all $\theta \in \p$ (cf. \cite[Theorem~7]{JDE_ensembles_2021}). However, Example~5 in \cite{JDE_ensembles_2021} shows that ensemble reachability (w.r.t. to  $L_n^q(\p)$) does not imply pointwise reachability for all parameters $\theta$. So, when considering Lebesgue spaces $L_n^q(\p)$ the condition formulated in Definition~\ref{def:reachbility}\,$(ii)$ need to be interpreted for all $\theta$ with the exception of null sets.  Moreover, it is  shown in \cite[Theorem~3.1.1]{triggiani75}, that a pair $(A,B) \in C_{n,n}(\p)\times X_{n,m}(\p)$ is ensemble reachable with respect to $X_{n}(\p )$
if and only if
\begin{equation}
\label{eq:Kalman-ensemblereach}
	\operatorname{span}\left\{ \operatorname{Im} \mathcal{A}^k\mathcal{B} \, |\, k=0,1,2,3,...  \right\}
\end{equation}
is dense in $X_n(\p)$, where $\operatorname{Im}$ denotes the image.

\subsection*{Novelty and main contribution}

The recent contributions metioned in the introduction of this paper all deal with open-loop inputs. In this work, we consider feedback in combination with open-loop inputs of the form
\begin{equation*}
    u(t,x(t)) = \mathcal{K} x(t) + v(t)
\end{equation*}
given by a feedback operator $\mathcal{K}$ defined on $X_n(\p)$. We will distinguish between the cases where the range of the feedback operator $\mathcal K$ is finite and infinite dimensional.

\medskip
\textit{Feedback operator with finite dimensional range}. An essential assumption in ensemble control is that the inputs are independent of the parameters. In view of this fact, it is natural to consider feedback operators with finite dimensional range, i.e.
\begin{equation*}
	\mathcal{K} : X_n(\p) \to \C^m .
\end{equation*}
In this context, this work explores classical system-theoretic properties. More precisely, in Theorem~\ref{thm:feedback_invariant} we will show that ensemble reachability is invariant under ensemble feedback with finite dimensional range. This is a well-known property in mathematical systems theory. On the other hand, however,  to our surprise we show in Theorem~\ref{thm:stab} that such feedback operators are not appropriate for stabilization. We note that this holds for all infinite-dimensional families of linear systems defined on arbitrary Banach spaces.

\medskip
\textit{Feedback operator with infinite dimensional range}. Motivated by the limitations of feedback operators with finite dimensional range, we also treat the case of infinite-dimensional range. Given that necessary and sufficient conditions for ensemble controllability depend on the particular Banach space under consideration, we limit ourselves to the Banach space of continuous functions. Here we investigate feedback multiplication operators
of the form
\begin{equation*}
	\mathcal{K} :C_n(\p) \to C_m(\p),\quad {\mathcal K} f (\theta) = K(\theta)f(\theta), \qquad K \in C_{m,n}(\p).
\end{equation*}
This type of feedback operators was recently considered in \cite{laa_feedback}, where also the case $K(\theta)\equiv K \in \R^{m \times n}$ is considered for the controlled harmonic oscillator. A main tool in our analysis will be the Hermite indices of a pair $(A.B)$, which are recalled at the beginning of Section~\ref{sec:staticfeedback}. Under the assumption of constant Hermite indices (as a function of $\theta \in \p$) a concrete construction is provided that yields a parametrized version of Heymann's lemma, \cite{heymann1968comments}, \cite[Lemma 8.4.7]{hinrichsen2026mathematicalII}. This lemma famously provides a key simplification in the proof of the pole placement theorem and similarly a parametrized pole placement result is provided here. This result can then be used as a basis for a novel stabilization result. 
In Theorem~\ref{thm:feedback-hermite} we provide new sufficient conditions such that there is a $K \in C_{m,n}(\p)$ so that the continuous family
\begin{equation*}
(A+BK,B) \in C_{n,n}(\p) \times C_{n,m}(\p)
\end{equation*}
is uniformly ensemble reachable. We note that this result improves Theorem~3 in \cite{laa_feedback}, as Theorem~\ref{thm:feedback-hermite} does not require the application of transformations in the state-space and the input space. 

For the case of nonconstant Hermite indices we restrict our attention to parameter sets that are intervals or circles. Also in this case a parameterized version of Heymann's lemma is shown which relies largely on standard tools from algebraic geometry. For more general parameter sets, we believe more tools from homotopy theory or obstruction theory would have to be applied, see e.g. \cite[Chapter VI]{hu1963homotopy}, but this is beyond the scope of the present paper.

 \section{Preliminaries}
 \label{sec:prelims}

In this section we collect auxiliary results from the literature which will be used in the verification of the main results of this paper. Of the following two subsections, the first one is concerned with spectra of the matrix multiplication operator. The second section recalls the relevant stability properties of infinite dimensional linear differential equations.

\subsection{Matrix multiplication operators}

In this section we will show that for the matrix multiplication operator defined on the space of continuous function $C_n(\p)$ several spectra coincide. First we recall the precise definitions. For a Banach space $X$ and $\mathcal A\in L(X)$, the spectrum of   is given by the set
\begin{equation*}
    \sigma_{}(\mathcal A) = \{ \lambda \in \C \mid \lambda-\mathcal A \quad \text{is not invertible in $L(X)$} \}.
\end{equation*}
An operator $\A$ is said to be Fredholm if $\operatorname{dim} \operatorname{Ker}(\A)< \infty$ and $\operatorname{codim} \operatorname{Im}(\A)< \infty$. As in \cite[Ch.~IV, 1.20]{engelnagel} we call
\begin{equation*}
\sigma_{\operatorname{ess}}(\mathcal A) = \{ \lambda \in \C \mid \lambda-\mathcal A \quad \text{is not a Fredholm operator} \}
\end{equation*}
the essential spectrum  of $\mathcal A$ (in the sense of Kato\footnote{There are various different definitions for the essential spectrum, which do not coincide in general. See \cite[Section 1.4]{appell2004nonlinear} for a detailed discussion.}).  We note that for a Fredholm operator, the range $\operatorname{Im}(\A)$ is automatically closed.   Moreover, the approximate point spectrum is given by
\begin{equation*}
\sigma_{q}(\mathcal A) = \{ \lambda \in \C \mid  \, \exists\, (v_k)_{k \in \mathbb N} \subset X \text{ s.t. } \| v_k\| =1  \text{ and }  \lim_{k \to \infty} \| (\lambda I - \mathcal{A})v_k\| =0 \}.
\end{equation*}
The sequences $(v_k)_{k \in \mathbb N} \subset X$ appearing in the previous definition are called Weyl sequences. 
For a comprehensive overview of the spectra of bounded operators and the relationships between their various components, we refer to \cite[Chapter~1]{appell2004nonlinear} and the references therein.

We now concentrate on sets $\p$ that do not contain isolated points, or in other words that are perfect. In this case, we show that for multiplication operators on $X=C_n(\p)$, the spectrum and essential spectrum coincide. This result follows from the application of more general results of Banach algebras and their modules. To keep the paper self-contained we prefer to give a direct proof here. For the case of square integrable functions, this property is shown in \cite{hardt1996spectral}. We start with a preliminary lemma that might be of individual interest.

\begin{lemma}
\label{lem:Pperfect}
Let Assumption~\ref{ass:basicP} hold and
assume furthermore that $\p$ is perfect. 
Then for every $A\in C_{n,n}(\p)$
the kernel of $\mathcal{A}$ is either trivial or infinite-dimensional. In particular, if $\mathcal{A}$ is Fredholm, the kernel is trivial.
\end{lemma}

\begin{proof} 
The kernel of $\mathcal{A}$ is given by
\[
\ker(\mathcal{A}) = \{ f \in C_n(\p) \mid A(\theta) f(\theta) = 0 \ \forall \theta  \in \p \}.
\]
Assume now that $\ker \A \neq \{0\}$ and let 
$0\neq v \in \ker \A$. By continuity the set $U :=\{ \theta \in \p \;\vert \; v(\theta) \neq 0 \}$ is open in $\p$. As $\p$ does not have isolated points, $U$ contains infinitely many points. It follows that the space
\begin{equation*}
    C_c(\p, U) = \{ f \in C(\p) \ \vert \ \operatorname{supp} f\subset U \}
\end{equation*}
is infinite-dimensional.\footnote{Pick $x_1 \in U$ and $\varepsilon_1>0$ such that $\overline{B}_{\varepsilon_1}(x_1)\cap \p \subset U$ and $U_1 := U \setminus \overline{B}_{\varepsilon_1}(x_1) \neq \emptyset$. Note that by our assumptions $U_1$ contains infinitely many points as it is a nonempty open subset of $\p$. By Urysohn's lemma, \cite[Chapter 4, Lemma 4]{Kell75}, there is a function $f_1\in C(\p)$ such that $f_1(x_1)=1$ and $
\operatorname{supp} f_1 \subset \overline{B}_{\varepsilon_1/2}(x_1)$. Now repeat this process by choosing $x_2 \in U_1$ and $\varepsilon_2>0$ such that $\overline{B}_{\varepsilon_2}(x_2)\cap \p \subset U_1$ and $U_2 := U_1 \setminus \overline{B}_{\varepsilon_2}(x_2) \neq \emptyset$, etc. and repeat inductively. This creates a sequence of continuous functions with pairwise nonintersecting supports in $U$. These are linearly independent.}
Now for every $f\in C_c(\p, U)$ it follows that $fv\in \ker \A$, as $\A (fv) = f \A v = 0$. As $C_c(\p, U)$ is infinite-dimensional, it follows that $\ker \A$ is also infinite-dimensional.

If $\mathcal{A}$ is Fredholm, then by definition its kernel is finite-dimensional and thus by the previous argument equal to $\{0\}$. 
\end{proof}

\begin{proposition}\label{prop:spec-mult-operator}
	Let $\p$ be nonempty and compact. Then for every $A \in C_{n,n}(\p)$ the following statements hold.
    \begin{enumerate}
        \item[(i)] The spectrum of the matrix multiplication operator $\mathcal{A}\colon C_n(\p) \to C_n(\p) $ is given by 
        \begin{equation*}
            \sigma (\mathcal{A}) =  \bigcup_{\theta \in \p}  \sigma (A(\theta)).
        \end{equation*}
         \item[(ii)] If, in addition, $\p$ is perfect, then 
	\begin{align*}
		\sigma_{q}(\mathcal A) =	\sigma_{\operatorname{ess}}(\mathcal A) = \sigma (\mathcal{A}) =  \bigcup_{\theta \in \p}  \sigma (A(\theta)).
	\end{align*}         
    \end{enumerate}
    
\end{proposition}

\begin{proof} 
(i) This is a special case of \cite[Lemma 7.1]{nagelone}.
See also \cite[Examples~2.7.2]{hardt1996spectral}. 

(ii) By definition, we have $	\sigma_{\operatorname{ess}}(\mathcal A) \subset  \sigma (\mathcal{A})$ and $ \sigma_{q}(\mathcal{A}) \subset \sigma(\mathcal{A}) $. To complete the proof, it remains to show the reverse inclusions $	  \sigma (\mathcal{A}) \subset \sigma_{q}(\mathcal A) $ and  $	  \sigma (\mathcal{A}) \subset \sigma_{\operatorname{ess}}(\mathcal A)$.

To see that $\sigma(\mathcal{A}) \subset \sigma_q(\mathcal{A})$,   
let $\lambda_0 \in   \sigma (\mathcal{A})$. By the characterization in the last identity of the claim, there  are $\theta_0 \in \p$ and  $v_0 \in \C^n$ such that $\|v_0\|=1$ and 
\begin{equation*}
A(\theta_0)v_0 = \lambda_0 v_0.
\end{equation*}
Consider a sequence of functions $(f_k)_{k \in \N} \subset C^{\infty}( \mathbb C,[0,1])$ satisfying
\begin{equation}
\label{eq:vk-definition}
f_k(\theta_0)=1 \qquad \text{and} \qquad f_k(\zeta)=0 \quad \text{ for all } \zeta \in \C \text{ with } | \zeta - \theta_0| \geq \tfrac{1}{k}.
\end{equation}
Hence, the sequence    $(v_k)_{k \in \mathbb N} \subset C_n(\p)   $,  $v_k(\theta):= f_k(\theta)v_0$, $\theta \in \p$, satisfies $\|v_k\|_{C_n(\p)} =1$. To see that $(v_k)_{k \in \mathbb N}$ is a Weyl sequence for $\lambda_0$, we have to verify that 
 $$
 \lim_{k \to \infty} \| (\lambda_0I -\mathcal{A})v_k\|_{C_n(\p)} =0.
 $$
 Fix $\e>0$.  For $k\in \N$ define  $\p_k:=\{  \theta \in \p \mid |\theta - \theta_0| < \frac{1}{k}\}$. Then, by construction we have that
 $$
  \| (\lambda_0I -\mathcal{A})v_k\|_{C_n(\p)} =
  \max_{\theta \in \p_k} \| f_k(\theta)(\lambda_0 v_0 - A(\theta)v_0) \| \leq \max_{\theta \in \p_k} \| \lambda_0 v_0 - A(\theta)v_0 \|.
  $$
Since the function $g(\theta) =  \| \lambda_0 v_0 - A(\theta)v_0 \|$ is continuous on $\p$ and satisfies $g(\theta_0)=0$, there is a $\delta >0$ (depending on $\varepsilon$) such that 
$$
g(\theta) =  \| \lambda_0 v_0 - A(\theta)v_0 \|  < \e \qquad \text{ for all } \quad | \theta - \theta_0| < \delta.$$
Hence,  we conclude that
$$ \max_{\theta \in \p_k} \| \lambda_0 v_0 - A(\theta)v_0 \| < \e \qquad \text{ for all }k > \frac{1}{\delta}.$$
As $\varepsilon>0$ was arbitrary, this shows that $\lambda_0 \in \sigma_q(\A)$. 

To show that $\sigma(\mathcal{A}) \subset \sigma_{\operatorname{ess}}(\mathcal{A})$, let $\lambda_0 \in   \sigma (\mathcal{A})$ and suppose to the contrary that $\lambda_0 \not \in \sigma_{\operatorname{ess}}(\mathcal{A})$. Then, by definition,  the operator $\lambda_0 I - \mathcal{A}$ is Fredholm. In particular, its range is closed and by Lemma~\ref{lem:Pperfect} its kernel is $\{0\}$. In other words, $\lambda_0 I - \mathcal{A}$ is injective and we may apply \cite[Theorem 3.12]{schechter2002principles}. This result shows that there is a constant $C>0$ such that
$$
\| v\|_{C_n(\p)} \leq C \| (\lambda_0 I- \mathcal{A})v\|_{C_n(\p)}
$$
for all $v \in C_n(\p)$. Considering the sequence $v_k$ as defined in \eqref{eq:vk-definition} and choosing $k$ sufficiently large, we obtain a contradiction. This shows the assertion.
\end{proof}

\begin{remark}
     The statement of Proposition~\ref{prop:spec-mult-operator}\,(ii) is evidently false without the assumption that $\p$ is perfect. In this case, an isolated point $\theta^*$ in $\p$ gives rise to eigenvalues for instance for $A$ given by $A(\theta)=0$, $\theta \in \p\setminus \{\theta^*\}$ and $A(\theta^*)\neq 0$. 
\end{remark}

\subsection*{Stability of linear infinite dimensional systems}{}

\phantom{f}

Let $A\in C_{n,n}(\p)$ and $\mathcal{A}$ be the associated mulitplication operator on $X_n(\p)$.
In order to study the stabilization problem for the class of infinite dimensional systems given by \eqref{eq:sys-inf2} and \eqref{eq:sys-inf-discrete2}, it is required to consider the stability properties of the uncontrolled system
\begin{align}
	\begin{split}
		\label{eq:sys-infinite-open}
		\dot x (t)&=  \mathcal{A} x(t)\\
		x(0)&=x_0.
	\end{split}
\end{align}
The multiplication operator $\mathcal{A}$ generates the semigroup
\begin{align}\label{eq:T-semi}
	T(t)\colon X_n(\p) \to X_n(\p),\quad T(t)f (\theta) := e^{tA(\theta)} f(\theta).
\end{align}
It follows from \cite[Chapter~1, Theorem~3.7]{engelnagel}  that the semigroup $\{T(t)\}_{t\geq0}$ is uniformly continuous, i.e. the map $t\mapsto T(t)$ is continuous with respect to the uniform operator topology.  Moreover, from \cite[Chapter~1, Section~3]{engelnagel} we deduce the following characterization of stability for system \eqref{eq:sys-infinite-open}.

\medskip

\begin{proposition}
	\label{prop:stability}
	Let $A \in C_{n,n}(\p)$ and let $\mathcal A$ denote the corresponding matrix multiplication operator defined on a Banach space $X_n(\p)$. Then, for the linear system \eqref{eq:sys-infinite-open} the following statements are equivalent.
	\begin{enumerate}
		\item The origin is uniformly asymptotically stable, i.e. 		one has
		$$ \lim_{t \to \infty} \| T(t) \| = 0.$$
		\item The origin is exponentially stable, i.e. there are constants $M\geq 1$ and $\gamma >0$ such that $\|T(t)\| \leq M e^{-\gamma t}$.
		\item The spectrum of $\mathcal{A}$ is contained in the open left half plane, i.e.
		$$\sigma( \mathcal{A}) \subset \{ z \in \C \mid \operatorname{Re}z < 0 \}.$$
	\end{enumerate}
\end{proposition}

Note that \cite[Ch.~I, Sec.~3, Exercise~4.8]{engelnagel} gives an example of a multiplication operator on $C_0(\R,\C)$ which generates an asymptotically stable semigroup in the sense that
$$ \lim_{t \to \infty} \| T(t)x_0\| = 0$$
for all $x_0\in C_0(\R,\C)$, but the semigroup is not uniformly asymptotically stable.

\section{Static state feedback for families on Banach spaces}
\label{sec:staticfeedback}

In this section we continue the systematic investigation of feedback methods for systems described by \eqref{eq:sys-inf2} or \eqref{eq:sys-inf-discrete2}.
We will distinguish two possible ways of implementing a linear static state feedback. We speak of {\it ensemble feedback operators}, if the feedback is given by a bounded linear operator
\begin{equation}
	\label{eq:ensemble-feedback-operator} 
	\mathcal{F} :X_n(\p) \to \C^m.
\end{equation}
More generally, we also consider {\it multiplication feedback operators} given by multiplication operators, i.e. by linear, bounded feedback operators defined by $K \in C_{m,n}(\p)$ and which are
of the form
\begin{equation}
	\label{eq:sys-infinite-feedback-multiplication} \mathcal{K} :C_n(\p) \to C_m(\p),\quad ({\mathcal K} x) (\theta) = K(\theta)x(\theta), \quad  x\in C_n(\p), \theta \in \p.
\end{equation}

Recall that a central point in the theory of ensemble reachability is that the input $u$ does not depend on the parameter $\theta\in \p$. Rather, it serves as an input applied simultaneously for all parameters. From this point of view the feedback operators satisfying \eqref{eq:ensemble-feedback-operator} appear to be more natural.

We emphasize that, because $\mathcal F$ has finite dimensional range, ensemble feedback operators are  automatically compact, cf. \cite[Theorem~4.10]{bressan-LN-functional-analysis}. Natural choices for such feedback operators might be given by the integral operator with a kernel $K\in C_{m,n}(\p)$ defined by
\[ \mathcal F f = \int_{\p} K(\theta) f(\theta) \di{\theta}, \qquad f\in C_n(\p),\]
or the weighted average operator (for fixed $\theta_1,\ldots,\theta_N\in\p$, $K_1,\ldots,K_N\in \C^{m\times n}$)
\[ \mathcal F f = \sum_{k=1}^N  K_k f(\theta_k), \qquad   f\in C_n(\p) .\]
Note that, the latter serve as examples and the subsequent analysis is not limited to these choices.

\subsection{Ensemble feedback}

Let $X_n(\p)$ denote a separable Banach space of functions from the parameter space $\p$ to $\C^n$ such that Assumption~\ref{ass:basicP} holds. Thus the multiplication operators $\mathcal A \colon X_n(\p) \to X_n(\p)$ and $\mathcal{B} \colon \C^m \to X_n(\p)$ are bounded linear. Then, using an bounded linear ensemble feedback operator $\mathcal F \colon X_n(\p) \to \C^m$, the overall systems is given by
\begin{equation}
	\begin{split}
		\label{eq:sys-infinite-feedback}
		\dot x (t)&= \left( \mathcal{A} +  \mathcal{B} \, \mathcal F \right) x(t)+  \mathcal{B}u(t)
		\end{split}
\end{equation}
The first result is obtained by using arguments of \cite[Chapter~8]{hinrichsen2026mathematicalII}.

\medskip

\begin{theorem}[Ensemble reachability is invariant under ensemble feedback]\label{thm:feedback_invariant}
    Let $X_{n}(\p)$ satisfy Assumption~\ref{ass:basicP}. Consider $A \in C_{n,n}(\p)$ and $B\in X_{n,m}(\p)$ with associated bounded linear operators $\mathcal{A} \in L(X_n(\p))$, $\mathcal{B}\in L(\C^m, X_n(\p))$. Then for every $\mathcal{F}\in L(X_n(\p),\C^m)$ the pair $(\mathcal{A},\mathcal{B})$ is ensemble reachable if and only if $(\mathcal{A}+ \mathcal{BF},\mathcal{B})$ is ensemble reachable.
\end{theorem}

\medskip
\begin{proof}
    Using the characterization of ensemble reachability in \eqref{eq:Kalman-ensemblereach},
	we shall show that for any ensemble feedback operator $\mathcal F \in L(X_n(\p),\C^m)$ it holds that
	\begin{equation}
    \label{eq:kalman-proof1}
		\operatorname{span}\left\{ \operatorname{Im} \mathcal{A}^k\mathcal{B} \, |\, k=0,1,2,3,...  \right\} =  \operatorname{span}\left\{ \operatorname{Im} (\mathcal{A}+ \mathcal{B}\mathcal{F})^k\mathcal{B} \, |\, k=0,1,2,3,...  \right\}.
	\end{equation}
	We first show $\supseteq$ by induction. To see this, we verify that for all $k=1,2,3,...$ one has
	\begin{equation*}
		(\mathcal{A}+ \mathcal{B}\mathcal{F})^k ~ \operatorname{Im}\mathcal{B} \ 
        \subseteq \
		\mathcal{A}^k \operatorname{Im}\mathcal{B} + \cdots + \operatorname{Im}\mathcal{B}.
	\end{equation*}
	 Let $k=1$ and $f \in  \operatorname{Im}\mathcal{B}$. Then, as $\mathcal{F}f \in \C^m$ we get
	\begin{equation*}
		(\mathcal{A}+ \mathcal{B}\mathcal{F})  f =  \mathcal{A}f +  \mathcal{B} \mathcal{F} f   \in
		\mathcal{A} \operatorname{Im}\mathcal{B}  + \operatorname{Im}\mathcal{B}.
	\end{equation*}
	Supposing the claim is true for some $k \in \mathbb{N}$ we consider $k+1$, i.e. 
	\begin{align*}
		(\mathcal{A}+ \mathcal{B}\mathcal{F})^{k+1}  f &=
		(\mathcal{A}+ \mathcal{B}\mathcal{F}) (\mathcal{A}+ \mathcal{B}\mathcal{F})^{k}  f = (\mathcal{A}+ \mathcal{B}\mathcal{F}) g\\
		&= \mathcal{A} g + \mathcal{B} \mathcal{F} g,
	\end{align*}
	for some $g \in \mathcal{A}^k \operatorname{Im}\mathcal{B} + \cdots + \operatorname{Im}\mathcal{B}$. Using again that  $\mathcal{F}g \in \C^m$ it follows that $\mathcal{B} \mathcal{F} g \in \operatorname{Im} \mathcal{B}$. Hence, we obtain
	\begin{equation*}
		(\mathcal{A}+ \mathcal{B}\mathcal{F})^k ~ \operatorname{Im}\mathcal{B}  =
		\mathcal{A}^k \operatorname{Im}\mathcal{B} + \cdots + \operatorname{Im}\mathcal{B}.
	\end{equation*}
	This completes the induction and $\supseteq$ is shown for \eqref{eq:kalman-proof1}.
	
	To see that in \eqref{eq:kalman-proof1} also $\subseteq$ holds, it is sufficient to apply $\supseteq$ to
	$\tilde {\mathcal{A}}:= \mathcal{A} + \mathcal{B} \mathcal{F}$ and
	$ \tilde{\mathcal{F} }:= - \mathcal{F}$.
\end{proof}

\begin{remark}
    As it is easy to miss, we point out that the finite-dimensionality of $\mathcal{F}$ is used explicitly when we need to conclude that $\mathcal{BF}g$ lies in the image of $\mathcal{B}$. In comparison, we note that
    it will be a consequence of Theorem~\ref{thm:feedback-hermite} that ensemble reachability is not an invariant under multiplication feedback.
\end{remark}

We now turn to the stabilization problem and recall the definition, cf. \cite{pritchard1981stability,triggiani1975stabilizability}.

\medskip
\begin{definition}
	A pair $(A,B) \in C_{n,n}(\p) \times X_{n,m}(\p)$ is called (ensemble) stabilizable if there is exists a linear bounded ensemble feedback operator $\mathcal{F} \colon X_n(\p) \to \C^m$ such that the closed-loop system defined by $\mathcal{A+BF}$ is exponentially stable.
\end{definition}

\medskip

We note that this definition is in line with the classical infinite-dimensional literature. There it is common to demand the existence of a bounded linear feedback operator so that the closed-loop system is exponentially stable, cf. \cite[Definition~8.1.1]{curtainzwart2020}, \cite[Ch.~VI, Definition~8.23]{engelnagel}. This is because in infinite-dimensions uniform asymptotic stability (or equivalently exponential stability) is in general not characterized in terms of the spectrum. cf. \cite{triggiani1975pathological}.

\medskip
\begin{theorem}\label{thm:stab}
Let Assumption~\ref{ass:basicP} hold and
assume furthermore that $\p$ is perfect. 
	Suppose the pair $(A,B) \in C_{n,n}(\p) \times X_{n,m}(\p)$ is not exponentially stable. Then, it is not exponentially stabilizable by  ensemble feedback, i.e., for any bounded linear ensemble feedback operator $\mathcal F \colon X_n(\p) \to \C^m$ the closed-loop system $\mathcal{A+BF}$ is not exponentially stable.
\end{theorem}

\begin{proof}
	As the range of $\mathcal{F}$ is finite dimensional, the operators $\mathcal F$ and $\mathcal{B} \mathcal{F}$ are compact, cf. \cite[Ch.VI, \S3, Prop.~3.5~(b)]{conway_functional_2nd}.
	As the essential spectrum of a $\mathcal{A}$ is invariant under compact perturbations, cf.~\cite[Thm.~5.10]{schechter2002principles}, we conclude from Proposition~\ref{prop:spec-mult-operator} that
	$$
	\sigma (\mathcal{A}) = \sigma_{\operatorname{ess}}(\mathcal A) =
	\sigma_{\operatorname{ess}}(\mathcal{A+BF}). $$
	The assertion then follows from Proposition~\ref{prop:stability}.
\end{proof}

\begin{remark}
    We note that Theorem~\ref{thm:stab} is false for trivial reasons, when the assumption is dropped that $\p$ does not have isolated points. If $q\in \p$ is isolated, then it may happen that $A_{\vert\p \setminus \{q\}}$ defines an exponentially stable system and it suffices to stabilize the finite-dimensional pair $(A(q),B(q))$. This is easily possible without destroying stability properties away from $q$ by choosing $K(\theta) = 0$, $\theta \in \p\setminus \{q\}$.
\end{remark}

\subsection{Multiplication feedback}

Motivated by the disappointing result of Theorem~\ref{thm:stab}, we now intend to study a different class  of feedback operators. The second contribution of this paper is to show that feedbacks using multiplication operators can be used for stabilization. For ensembles defined in the space of continuous functions, we will study {\it multiplication feedback operators} given by multiplication operators, i.e. feedback operators of the form
\begin{equation*}
 \mathcal{K} :C_n(\p) \to C_m(\p),\quad ({\mathcal K} x )(\theta) = K(\theta)x(\theta),
\end{equation*}
where $K \in C_{m,n}(\p)$. This class of feedback operators was recently considered in \cite{laa_feedback}, where also the case $K(\theta)\equiv K \in \R^{m \times n}$ is considered for the controlled harmonic oscillator. In this section we will provide (Theorem~\ref{thm:feedback-hermite}) new sufficient conditions so that a continuous ensemble becomes uniformly ensemble reachable by applying a mixture of open-loop inputs and multiplication feedback operators of the form
$$u(t,x(t,\theta)) = u(t) + K(\theta) x(t,\theta), \qquad K \in C_{m,n}(\p).$$
We note that Theorem~\ref{thm:feedback-hermite} in this paper improves a recent result in \cite[Theorem~3]{laa_feedback}, as it does not require the application of transformations in the state-space and the input space.

A crucial step in the construction procedure is an ensemble version of the classical  Heymann Lemma. To this end, we recall some facts concerning Hermite indices. For a reachable pair $(A,B)\in \R^{n\times n} \times \R^{n \times m}$, we consider the list (a permutation of the columns the Kalman matrix of $(A,B)$)
\begin{align*}
	\begin{matrix}
		b_1 & Ab_1& \cdots & A^{n-1}b_1&\cdots& b_m & Ab_m& \cdots & A^{n-1}b_m.
	\end{matrix}
\end{align*}
Then, we select from left to right the first linearly independent columns
\begin{equation*}
	\label{eq:def:Hermiteindicesselection}
	b_1,\ldots,A^{h_1-1}b_1,b_2, \ldots, A^{h_2-1}b_2, \ldots ,b_m,\ldots,A^{h_m-1}b_m.
\end{equation*}
The corresponding exponents $h_1,...,h_m$ are called the \emph{Hermite indices} of $(A,B)$, where $h_i:=0$ if  $b_i$ is not selected, see also \cite[Scheme~II, Sec.~6.4.6]{kailath1980}. Note that the Hermite indices of a pair $(A,B)$ are not feedback invariant, in contrast to the Kronecker and the controllability indices, cf. e.g. \cite{Bar-Fer-Zab-2007-generic,hinrichsen2026mathematicalII}. In this context, we note that \cite[Thm.~2.1]{Bar-Fer-Zab-2007-generic} characterizes the possible Hermite indices that can be assigned through feedback in terms of the Kronecker indices of a system.  

The following lemma generalizes a frequently-used preparatory step in the proof of Heymann’s lemma to the ensemble case; see \cite{heymann1968comments} for the original statement. We need the following notation: For a given reachable pair $(A,B)\in \R^{n\times n} \times \R^{n \times m}$ with Hermite indices $h_1, ...,h_m$, and the convention $h_0:=0$, we  set
\begin{equation}
\label{eq:Hj-definition}
    H_j:= \sum_{i=0}^{j}   h_i, \quad j=0,\ldots, m.
\end{equation}

\begin{lemma}\label{lem:Heyman-Hermite}
	Let $\p$ be nonempty and compact. Let $(A,B) \in C_{n,n}(\p) \times C_{n,m}(\p)$ be pointwise reachable and suppose that the Hermite indices are constant. Then, there are $v_1,...,v_{n} \in \C^m$ and $x_1,\ldots,x_n \in C_n(\p)$ defined by setting $x_0(\theta):=0, \theta\in\p$, and 
	\begin{equation}
    \label{eq:xk-definition}
		x_{k}(\theta) := A(\theta)x_{k-1}(\theta) + B(\theta)v_{k}, \qquad \theta \in \p, k=1,\ldots,n,
	\end{equation}
    such that for all $\theta\in \p$ the vectors $x_1(\theta),...,x_n(\theta) \in \C^n$ 
	are linearly independent in  $\C^n$.
\end{lemma}

\medskip

\begin{proof}
	Let $h_1,...,h_m$ denote the Hermite indices. Then,  for every $\theta\in\p$ the vectors
	\begin{align}\label{eq:Hermite-assumption}
		b_1(\theta),...,A(\theta)^{h_1-1}b_1(\theta),...,b_m(\theta),...,A(\theta)^{h_m-1}b_m(\theta)
	\end{align}
	are linearly independent in  $\C^n$.   Using the notation introduced in \eqref{eq:Hj-definition}, we set the vectors $v_1,...,v_n$ as
    \begin{equation*}
        v_k := \begin{cases}
		 e_{l} & \text{ if } k = H_{l-1} + 1 \text{ and } H_{l-1} < H_l ,\\
		 0 & \text{ else.}
		\end{cases}
    \end{equation*}
As the Hermite indices are independent of $\theta$ this causes no problems. Also, without loss of generality we assume that $b_1\neq 0$, otherwise all the indices in the following arguments are shifted to the first nonzero column. Then $H_0=0$, $H_1=h_1 >0$ and so $v_1 = e_1$, $v_2=\ldots =v_{h_1-1}=0$.
    By construction, we have
	\begin{align}\label{eq:construct-v}
		\begin{split}
			x_1&= Ax_0 + B v_1 = Be_1 = b_1 \\
			x_2&= Ax_1 + Bv_2= Ab_1\\
			& \,\,\,\vdots\\
			x_{h_1}&= A^{h_1-1}b_1.
		\end{split}
	\end{align}
    The list $(x_1,\ldots,x_{h_1})$ is linearly independent by the definition of the Hermite indices. For the next index, $h_1+1 = H_{1}+1$, consider the next nonzero Hermite index in the list $(h_1,\ldots,h_m)$, say $h_{l_2}$. This means $H_1 = H_2 = \ldots = H_{l_2-1}< H_{l_2}$ and $h_1+1 = H_{l_2-1}$. Hence $v_{h_1+1} = e_{l_2}$ and $v_{h_1+2} = \ldots = v_{H_{l_2}-1} = 0$. We thus obtain
		\begin{align}
	\begin{split}
			x_{h_1+1}&= A^{h_1}b_1 + Bv_{h_1+1} = A^{h_1}b_1 + b_{l_2} \\
		x_{h_1+2}&= A^{h_1+1}b_1 + Ab_{l_2} + Av_{h_1+2}= A^{h_1+1}b_1 + Ab_{l_2}\\
			& \,\,\,\vdots\\
		x_{h_1+h_{l_2}}&= A^{h_1+h_{l_2}-1}b_1 + A^{h_{l_2}-1}b_{l_2}.
		\end{split}
		\end{align}   
    As the vectors $(A^{h_1}b_1,\ldots,A^{h_1+h_{l_2}-1}b_1)$ are in the span of $(x_1,\dots,x_{h_1})$ by the construction of the Hermite indices, we obtain that $(x_1,\ldots,x_{H_{l_2}})$ are linearly independent.
    This procedure now repeats by choosing the next nonzero Hermite index. All steps are identical and the procedure terminates in a basis $(x_1,\ldots,x_n)$ of $\C^n$.
\end{proof}

\medskip
As the vectors $x_1,...,x_n$ are defined by the pair $(A,B)$ and its Hermite indices, we will use the notation
\begin{equation*}
	M_{A,B} :=\begin{pmatrix}    x_1&  x_2 & \cdots& x_n\end{pmatrix}.
\end{equation*}
For later use, we fix the following fact.

\begin{corollary}
	\label{cor:Heyman-Hermite-determinant}
	Let $\p$ be nonempty and compact. Let $(A,B)\in C_{n,n}(\p) \times C_{n,m}(\p)$ be pointwise reachable and suppose that the Hermite indices are constant. Then, for every $\theta\in \p$ the vectors $x_1(\theta),...,x_n(\theta)$ constructed in  Lemma~\ref{lem:Heyman-Hermite}  satisfy
	\begin{equation*}
		\det M_{A,B}(\theta) =
		\det \begin{pmatrix}
			x_1(\theta)& \cdots
			& x_n(\theta)
		\end{pmatrix}
		\neq 0.
		\end{equation*}
\end{corollary}

Now we are ready to prove the following parameter version of Heymann's Lemma.

\begin{lemma}[Parameter version of Heymann's Lemma]\label{lem:Heymann-parameter}
	Let $\p$ be nonempty and compact. Let $(A,B) \in C_{n,n}(\p) \times C_{n,m}(\p)$ be pointwise reachable and suppose that the Hermite indices are constant. Then, there are  $ K \in C_{m,n}(\p)$  and  $v \in \mathbb{C}^n$ such that the single-input pair
	\begin{align*}
		\left( A+B K, Bv\right) \in C_{n,n}(\p)\times C_n(\p)
	\end{align*}
	is pointwise reachable.
\end{lemma}

\begin{proof}
	Fix $v_1,...,v_{n} \in \C^m$ and  $x_1,...,x_n\in C_n(\p)$  as in Lemma~\ref{lem:Heyman-Hermite} and an arbitrary $v_{n+1}\in C_n(\p)$. Then, we define the matrix-valued function
	\[
	K: \p \to \R^{m \times n}, \theta \mapsto  K(\theta) = \operatorname{col}(v_2,...,v_{n},v_{n+1}(\theta))\,\operatorname{col}(x_1(\theta),...,x_n(\theta))^{-1}.
	\]
    By Corollary~\ref{cor:Heyman-Hermite-determinant} and since taking the inverse is a continuous operation,  we have $ K \in C_{m,n}(\p)$. By construction, $ K(\theta) x_k(\theta) = v_{k+1}(\theta)$, $k=1,\ldots,n$, and with \eqref{eq:xk-definition} it holds 
	\begin{align*} 
    B (\theta)v_1 &= x_1(\theta), \\
	\left(A(\theta)+B(\theta) K(\theta) \right)x_{k}(\theta)&=  A(\theta)x_{k}(\theta)+B(\theta) K(\theta) x_{k}(\theta) =
	A(\theta)x_{k}(\theta)+B(\theta) v_{k+1}(\theta)\\ &=
	x_{k+1}(\theta), \quad k= 1, \ldots, n-1.
	\end{align*} 
	Therefore
	\[
	\left(A(\theta)+B(\theta) K(\theta) \right)^k B(\theta)v_1
       = x_{k+1}(\theta), \quad k=1,...,n-1.
	\]
	Thus, for every $\theta \in \p$, the Kalman matrix of the single-input system given by $\left( A(\theta) +B(\theta) K(\theta), B(\theta)v_1\right)$ is given by
	\begin{align*}
		\begin{pmatrix}
			B(\theta)v_1\!&\!\left(A(\theta)+B(\theta) K(\theta)\right) B(\theta)v_1 &\!\!\cdots\!\!& \left(A(\theta)+B(\theta) K(\theta)\right)^{n-1} B(\theta)v_1
		\end{pmatrix} 
		=M_{A,B}(\theta).
			\end{align*}
	From Corollary~\ref{cor:Heyman-Hermite-determinant} and the Kalman rank criterion, we conclude that $\left( A +B K, Bv_1\right)$ is pointwise reachable.
\end{proof}

With this version of Heymann's lemma at hand, we can prove the pole placement theorem in the next statement. Before doing so, we need to point out a significant difference between Heymann's original lemma and Lemma~\ref{lem:Heymann-parameter}: In the original lemma, any nonzero $b\in \im B$ can be used to obtain a reachable pair $(A+BK,b)$. We will discuss extensions in this direction in Section~\ref{sec:homotopy}.

\begin{theorem}
[Parameter-dependent pole placement]\label{lem:pole-placement-parameter}
	Let $\p$ be nonempty and compact. Let $(A,B) \in C_{n,n}(\p) \times C_{n,m}(\p)$ be pointwise reachable and suppose that the Hermite indices are constant. Then, for every tuple of continuous functions $\lambda_1, ..., \lambda_n \colon \p \to \C$ there is a  matrix function $ K \in C_{m,n}(\p)$ such that
	\begin{equation}
    \label{eq:poleplacement}
		\sigma ( \mathcal{ A+B K} ) = \bigcup_{\theta \in \p} \{\lambda_1(\theta),...,\lambda_n(\theta)\}.
	\end{equation}
\end{theorem}

\begin{proof}
	Let  $\lambda_1, ..., \lambda_n \in C(\p)$. By Lemma~\ref{lem:Heymann-parameter}, there are a $\tilde K \in C_{n,m}(\p)$ and  a $v\in\C^m$ such that $(A+B\tilde K,Bv)$ is pointwise reachable. Consider the family of polynomials $\{ p_\theta \mid \theta \in \p\}$ given by
	\[
	p_\theta(z) := \big(z-\lambda_1(\theta)\big) \cdots \big(z-\lambda_n(\theta)\big)  = z^n-p_{n-1}(\theta) z^{n-1} + \cdots + p_1(\theta)z+p_0(\theta).
	\]
	Note that the coefficients $p_0,...,p_{n-1}$ depend continuously on the parameter $\theta$. With this at hand and using Ackermann's Formula, cf. \cite[Theorem~6.19]{fuhrmann2015mathematics}, we define
	\[
	\tilde k \in C_{1,n}(\p), \quad \tilde k(\theta) =
	\begin{pmatrix}0& \cdots & 0&1\end{pmatrix}\,
	T(\theta)^{-1} \, 
    p_\theta \big(A(\theta)+B(\theta)\tilde K(\theta)\big),\quad\theta\in\p,
	\]
	where $T(\theta)$ denotes the Kalman matrix corresponding to the single-input system $(A(\theta)+B(\theta)\tilde K(\theta),B(\theta)v) $, $\theta\in\p$. It then follows from Ackermann's formula that 
    \begin{equation*}
        \sigma (  A(\theta)+B(\theta) \tilde K(\theta) + B(\theta) v \tilde{k}(\theta) ) =
        \{ \lambda_1(\theta), \ldots, \lambda_n(\theta)\}, \quad \theta \in \p.
    \end{equation*}
    Defining $K\in C_{m,n}(\p)$ by
	\begin{equation*}
	 K(\theta):= \tilde K(\theta)+ v\tilde k(\theta), \quad \theta\in\p,   
	\end{equation*}
	  we obtain a continuous matrix function, for which
    \eqref{eq:poleplacement} holds by an application of Proposition~\ref{prop:spec-mult-operator}.
\end{proof}

\medskip

After finishing the preliminary steps, we are able to state the main results of this section. The first result says that the class of multiplication ensenmble feedback operators is suitable for exponential stabilization.

\medskip

\begin{theorem}\label{thm:feedback-exp-stab}
	Let $\p$ be nonempty and compact. Suppose $(A,B)\in C_{n,n}(\p) \times C_{n,m}(\p)$ is pointwise reachable and has constant Hermite indices.  Then, there is a multiplication ensemble feedback operator $\mathcal{K}:  C_{n}(\p) \to C_{m}(\p)$ such that $\mathcal{A+B K}$ is exponentially stable.
\end{theorem}

\begin{proof}
	   Choose continuous functions $\lambda_1, ..., \lambda_n \colon \p \to \{ z \in \C \mid \operatorname{Re}z<0\}$. Using Proposition~\ref{prop:spec-mult-operator} and Theorem~\ref{lem:pole-placement-parameter}, we conclude that there is a  $K \in C_{n,m}(\p)$ such that
	$$\sigma( \mathcal{A+BK}) = \bigcup_{\theta \in \p} \{\lambda_1(\theta),...,\lambda_n(\theta)\} \subset \{ z \in \C \mid \operatorname{Re}z<0\}.
	$$
	The assertion then follows from Proposition~\ref{prop:stability}. 
\end{proof}

\begin{remark}
 We emphasize that for the latter statement it is not required that the parameter set $\p$ is perfect, because it is not important for the result whether or not the entire spectrum is essential. A key feature is that the multiplication feedback operator is not compact. Because of this, it can be applied to shift the entire spectrum to the open left half plane. 
\end{remark}

\medskip
Our second main result, Theorem~\ref{thm:feedback-hermite}, is concerned with sufficient conditions for the existence of multiplication feedback operators so that the mixed open-loop and feedback system becomes uniform ensemble reachability. We note that the sufficiency conditions are verifiable just in terms of the matrices $A(\theta)$ and $B(\theta)$. The proof of Theorem~\ref{thm:feedback-hermite} will use a modified version of sufficient conditions for uniform ensemble reachability of single-input systems. The details are provided in the Appendix~\ref{sec:appendix}.

\begin{remark}
In \cite[Thm~3, Thm.~4, Cor.~3]{JDE_ensembles_2021} it is shown that uniform ensemble reachability is a subtle interplay between spectral properties of the matrix function $A$, system theoretic properties of the pairs  $(A(\theta),B(\theta))$ and topological properties of the parameter space $\p$. The significance of the next result is that,  if the system-theoretic property that the Hermite indices are constant holds, the spectral properties can be put in place by using not just open-loop inputs but a mixture $u(t) + K(\theta)x(t,\theta)$ of open-loop inputs and feedback.  Therefore, the next statement is of particular interest for the cases where a family of systems is to be steered (approximately) to a desired family of states rather than just  stabilizing it.   
\end{remark}

\medskip
	
	\begin{theorem}\label{thm:feedback-hermite}
		Let $\p$ be nonempty, compact with empty interior and $\C \setminus \p$ is connected. Suppose $(A,B)\in C_{n,n}(\p) \times C_{n,m}(\p)$ is pointwise reachable and has constant Hermite indices.  Then, there is a  $K\in C_{m,n}(\p)$ such that $(A+B K,B)$ is uniformly ensemble reachable.
	\end{theorem}
	
	\medskip

	\begin{proof}
		Choose continuous and injective functions $\lambda_1,...,\lambda_n: \p \to \C$ such that
		\[ \lambda_k(\p) \cap \lambda_l(\p) = \emptyset \qquad \text{for all } l \neq k.
		\]
		By Lemma~\ref{lem:Heymann-parameter}, Theorem~\ref{lem:pole-placement-parameter},  there is a $K(\cdot):= \tilde K(\cdot)+v\tilde k(\cdot)\in C_{n,m}(\p)$ such that the single-input pair
		$$
		\left( A(\theta)+B(\theta) K(\theta), B(\theta)v\right)
		$$
		is pointwise reachable and such that
        \begin{equation}
            \label{eq:lambdaforP51}
            \sigma(A(\theta)+B(\theta)K(\theta)) = \{\lambda_1(\theta),\ldots,\lambda_n(\theta)\}, \quad \theta \in \p.
        \end{equation}
        It remains to show that $(A+BK, B)$ is uniformly ensemble reachable.

		Since the vector $v$ does not depend on the parameter, it follows that $B(\theta) v \in \operatorname{Im} B(\theta)$. Hence, we have
		\begin{align*}
			\operatorname{span} &\left\{ \theta \mapsto \left(A(\theta) + B(\theta)\big(\tilde K(\theta)+v\tilde k(\theta)\big)\right)^k B(\theta)v \, \, |\, k=0,1,2,3,...\right\} \\
			&\subset
			\operatorname{span} \left\{ \theta \mapsto \left(A(\theta) + B(\theta) K(\theta)\right)^k b_j(\theta) \, \, |\,j=1,...m,\,  k=0,1,2,3,...\right\}.
		\end{align*}
		As \eqref{eq:lambdaforP51} is satisfied and the system $\left(A+B K  ,Bv\right) $ is pointwise reachable, we may apply Proposition~\ref{prop:1} in the appendix, and conclude that $\left(A+B K  ,Bv\right) $ is uniformly ensemble reachable, i.e.´
		\begin{align*}
			\operatorname{span} \left\{ \theta \mapsto \left(A(\theta) + B(\theta)\big(\tilde K(\theta)+v \tilde k(\theta)\big) \right)^k  B(\theta)v \, \, |\, k=0,1,2,3,...\right\}
		\end{align*}
		is dense in $C_n(\p)$ by \eqref{eq:Kalman-ensemblereach}. Consequently,
		\begin{align*}
			\operatorname{span} \left\{ \theta \mapsto \left(A(\theta) + B(\theta) K(\theta)\right)^k b_j(\theta) \, \, |\,j=1,...m,\,  k=0,1,2,3,...\right\},
		\end{align*}
		is also dense in $C_n(\p)$ and, by \eqref{eq:Kalman-ensemblereach}, $(A+BK,B)$ is uniformly ensemble reachable. This shows the assertion.
	\end{proof}

\begin{corollary}
Let $\p$ be nonempty, compact with empty interior and $\C \setminus \p$ be connected. Suppose $(A,B)\in C_{n,n}(\p) \times C_{n,m}(\p)$ is pointwise reachable and has constant Hermite indices. Then there is a $K\in C_{m,n}(\p)$ such that $A+B K$ is exponentially stable and $(A+B K,B)$ is uniformly ensemble reachable.
\end{corollary}

	\begin{proof}
        It suffices to choose the $\lambda_i : \p \to \C$ in the proof of Theorem~\ref{thm:feedback-hermite} such that they are continuous and injective functions whose images are in the open left-half plane and pairwise disjoint. 
	\end{proof}

\medskip
To end this section, we comment on the relation of Theorem~\ref{thm:feedback-hermite} to \cite[Theorem~3]{laa_feedback}. Apart from pointwise reachability, that is supposed in both results, the two theorems make an assumption on structure indices of the matrix pair $(A,B)$. 

More precisely, Theorem~\ref{thm:feedback-hermite} supposes that the Hermite indices are constant, whereas \cite[Theorem~3]{laa_feedback} assumes that the Kronecker indices are constant.

A central difference between the two indices is that the Kronecker indices are invariant under feedback, cf. \cite[Lemma~6.16]{fuhrmann2015mathematics}, whereas the Hermite indices are not invariant under feedback, cf. \cite{Bar-Fer-Zab-2007-generic}, \cite[Chapter 6.6, Exercise 23]{hinrichsen2026mathematicalII}.

Regarding the conclusions of the theorems, we note that Theorem~\ref{thm:feedback-hermite} provides the existence of a continuous feedback matrix $K$ such that $(A+BK,B$ is uniformly ensemble reachable. In contrast, \cite[Theorem~3]{laa_feedback} shows that $(A,B)$ is feedback equivalent to an uniformly ensemble reachable pair. More precisely, \cite[Theorem~3]{laa_feedback} proves that there are continuously invertible matrices $T \in C_{n,n}(\p)$ and $S \in C_{m,m} (\p)$ and a feedback matrix $K \in C_{n,m}(\p)$ such that the pair
\[
\left( T(A-BS^{-1}K)T^{-1}, TBS^{-1}\right)
\]
is uniformly ensemble reachable.

\section{Intervals and Circles as Parameter Spaces}
\label{sec:homotopy}

In this section we depart from the assumption of constant Hermite indices. In this case the constructions become less concrete, and we require further assumptions on the parameter space $\p$, but on the the other hand more general statements of Heymann's lemma can be obtained. We start with a general observation.

Let $\p$ be a compact parameter space and let $v\in\C^m$ be such that
\[
B(\theta)v \neq 0, \qquad \theta\in \p.
\]
For each $\theta\in \p$, consider the set
\begin{equation}
\label{eq:Etheta}    
E_\theta
:=
\Bigl\{
F\in\C^{m\times n}
:
(A(\theta)+B(\theta)F,\; B(\theta)v)
\text{ is reachable}
\Bigr\}.
\end{equation}
If we assume that $(A,B)$ is pointwise reachable, then, by Heymann's lemma, $E_\theta$ is nonempty for every $\theta\in P$. More precisely, if we consider the determinant of the Kalman matrix
\[
\Phi(\theta,F)
:=
\det\!\bigl[
B(\theta)v,\,
(A(\theta)+B(\theta)F)B(\theta)v,\,
\dots,\,
(A(\theta)+B(\theta)F)^{n-1}B(\theta)v
\bigr],
\]
then $E_\theta=\{F:\Phi(\theta,F)\neq 0\}$, hence $E_\theta$ is the complement of a proper algebraic subset of $\C^{m\times n}$. In particular, $E_\theta$ is dense, open and path connected. We define the \emph{Heymann bundle}
\begin{equation}
\label{eq:Heymannbundle}
E
:=
\bigl\{
(\theta,F)\in \p\times\C^{m\times n}
:
F\in E_\theta
\bigr\}.
\end{equation}
Since $\Phi$ is continuous, $E$ is an open subset of the trivial bundle
$\p\times\C^{m\times n}$, and the projection
\[
\pi:E\to \p,\qquad (\theta,F)\mapsto \theta,
\]
is continuous.

The problem of constructing a parameter-dependent feedback may now be formulated as a section problem: does there exist a continuous map
\[
s:\p\to E
\]
such that $\pi\circ s=\operatorname{id}_\p$? Writing $s(\theta)=(\theta,K(\theta))$, this is equivalent to finding a continuous map
\[
K:\p\to\C^{m\times n}
\]
for which
\[
(A(\theta)+B(\theta)K(\theta),\,B(\theta)v)
\]
is reachable for every $\theta\in \p$.

The existence of such a section is naturally related to the topology of the parameter space $\p$ and touches questions of homotopy theory and obstruction theory. A detailed analysis of this context is beyond the scope of this paper. However, we can present a simple case in which $\p$ is an interval, circle or homeomorphic image thereof.

With this we obtain a further parameter dependent version of Heymann's lemma.

\begin{lemma}[Parameter-dependent Heymann lemma]
\label{lem:Heymann-generic}
	Let $\p$ be a compact interval or a circle, or a homeomorphic image thereof. Let $(A,B) \in C_{n,n}(\p) \times C_{n,m}(\p)$ be pointwise reachable. Then, for every $v \in \C^{m}$ with the property that $B(\theta)v\neq 0$ for all $\theta \in \p$,
    there exists $ K \in C_{m,n}(\p)$  such that the single-input pair
	\begin{align*}
		\left( A+B K, Bv\right) \in C_{n,n}(\p)\times C_n(\p)
	\end{align*}
	is pointwise reachable.   
\end{lemma}

\begin{proof}
    We will construct a suitable finite open cover of intervals with associated constant feedback values for each open set in the cover. On the intersections these values will be patched. Without loss of generality we let $\p=[0,1]$. The case of a circle is treated in the same way as will be described.
    Concretely, using the notation introduced in \eqref{eq:Etheta} and \eqref{eq:Heymannbundle}:

For every $\theta \in \p$ choose $F_\theta\in E_\theta$. As $\Phi$ varies continuously in $\theta$, there is a (relatively) open interval $U_\theta$ around $\theta$ such that $F_\theta\in E_\eta$ for all $\eta\in U_\theta$. Choose a finite open subcover $\{U_1,\ldots,U_k\}$ of this open cover of $\p$ and denote the associated feedback values by $F_1,\ldots,F_k$. As $\p=[0,1]$, we may assume that the $U_i$ are ordered in such a way, that $U_i\cap U_{i+1} \neq \emptyset$, $i=1,\ldots,k-1$ and by reducing the size of $U_i$ if necessary, we may assume that $U_i\cap U_j = \emptyset$ if $|i-j|\geq 2$. For an index $i$, fix $\eta\in U_i\cap U_{i+1}$. Then $F_i,F_{i+1} \in E_{\eta}$ and as each fibre is connected there exists a continuous path $\gamma:[0,1]\to E_{\eta}$ with $\gamma(0)=F_i$, $\gamma(1)=F_{i+1}$. By continuity of $\Phi$, there exists an $\varepsilon >0$ such that $[\eta-\varepsilon,\eta+\varepsilon]\subset U_i\cap U_{i+1}$ and $\gamma([0,1])\subset E_\vartheta$ for all $\vartheta \in [\eta-\varepsilon,\eta+\varepsilon]$.
Now define a continuous function $K:U_i\cup U_{i+1}\to \C^{m\times n}$ by setting
\begin{equation}
    \begin{aligned}
        K(\theta) = \left\{ \begin{matrix}
            F_i &\quad &\theta \in (\inf U_i,\eta-\varepsilon), \\
            \gamma\left( \frac{1}{2\varepsilon}\theta - \frac{\eta-\varepsilon}{2\varepsilon} \right) & &\theta \in [\eta-\varepsilon,\eta+\varepsilon],\\
            F_{i+1} && \theta \in (\eta +\varepsilon, \sup U_{i+1}).
        \end{matrix}\right.
    \end{aligned}
\end{equation}
By construction $K$ is continuous and $K(\theta) \in E_\theta$ for all $\theta \in U_i\cup U_{i+1}$. It is clear that this construction may be performed on each of the finitely many overlaps $U_i\cap U_{i+1}$. In this way the desired continuous $K$ is constructed. 
\end{proof}

\begin{remark}
    In comparison the two different versions of Heymann's lemma, that is, Lemma~\ref{lem:Heyman-Hermite} and \ref{lem:Heymann-generic}, have different strengths and weaknesses. While Lemma~\ref{lem:Heyman-Hermite} asserts the existence of a vector $v$ for which the reduction of input dimension can be performed, the construction of the appropriate feedback only works for this particular vector and does not readily extend to other potential candidates $v$ with $B(\theta)v\neq 0$ everywhere. On the other hand Lemma~\ref{lem:Heymann-generic} shows the existence of a suitable feedback for every $v$ with nonvanishing image under $B(\cdot)$ but it is not shown that such a $v$ exists. Note, in addition, that the proof of Lemma~\ref{lem:Heymann-generic} also works for $v\in C_m(\p)$ but we have refrained from assuming this as for the application in the proof of Theorem~\ref{thm:feedback-hermite2} a constant $v$ is required.
\end{remark}

The following results can now be proved just as Theorems~\ref{lem:pole-placement-parameter}, \ref{thm:feedback-exp-stab}, \ref{thm:feedback-hermite} before.

\begin{theorem}
[Parameter-dependent pole placement]\label{lem:pole-placement-parameter2}
	Let $\p$ be a compact interval or a circle. Let $(A,B) \in C_{n,n}(\p) \times C_{n,m}(\p)$ be pointwise reachable and assume that there exists a $v\in \C^m$ with $B(\theta)v\neq0$ for all $\theta\in\p$. Then, for every tuple of continuous functions $\lambda_1, ..., \lambda_n \colon \p \to \C$ there is a  matrix function $ K \in C_{n,m}(\p)$ such that
	\begin{equation}
    \label{eq:poleplacement2}
		\sigma ( \mathcal{ A+B K} ) = \bigcup_{\theta \in \p} \{\lambda_1(\theta),...,\lambda_n(\theta)\}.
	\end{equation}
\end{theorem}

\begin{theorem}\label{thm:feedback-exp-stab2}
	Let $\p$ be a compact interval or a circle. Suppose that the pair $(A,B)\in C_{n,n}(\p) \times C_{n,m}(\p)$ is pointwise reachable and that there exists a $v\in \C^m$ with $B(\theta)v\neq0$ for all $\theta\in\p$.  Then, there is a multiplication ensemble feedback operator $\mathcal{K}:  C_{n}(\p) \to C_{m}(\p)$ such that $\mathcal{A+B K}$ is exponentially stable.
\end{theorem}

\begin{theorem}\label{thm:feedback-hermite2}
		Let $\p$ be a compact interval. Suppose $(A,B)\in C_{n,n}(\p) \times C_{n,m}(\p)$ is pointwise reachable and that there exists a $v\in \C^m$ with $B(\theta)v\neq0$ for all $\theta\in\p$.  Then there exists a  $K\in C_{m,n}(\p)$ such that $(A+B K,B)$ is uniformly ensemble reachable.
	\end{theorem}

\section{Conclusions}
\label{sec:conclusions}

This paper considers feedback methods for one-parameter families of linear systems, where the domain of the input operator is finite-dimensional. The parameter space is assumed to be compact subset in the complex plane. In Theorem~\ref{thm:stab} we show that feedback operators with finite-dimensional range are not appropriate to stabilize such an unstable family of linear systems. To overcome this limitation, we consider multiplication feedback operators for families of linear systems. In particular, we investigate families of linear systems defined over the space of continuous functions.  In this context, in Lemma~\ref{lem:Heymann-parameter} we prove a parameter-version of Heymann's Lemma for continuous families with constant Hermite indices. Based on this, in Theorem~\ref{thm:feedback-exp-stab} we show that pointwise reachable continuous families of linear systems can be exponentially stabilized by a multiplication feedback operator if the family has constant Hermite indices.
Moreover, it is shown Theorem~\ref{thm:feedback-hermite} that a pointwise reachable family with constant Hermite indices can be turned into an uniformly ensemble reachable family by a multiplication feedback operator if the complement of the parameter space in connected. In the case of nonconstant Hermite indices a parameterized version of Heymann's lemma has been obtained in the case that the parameter set is an interval or a circle. Using the tools presented in that section generalizations to more complicated case should be possible, e.g. for parameter sets of covering dimension $1$.

\section{Appendix}
\label{sec:appendix}

The next result is a mild refinement of \cite[Corollary~4]{JDE_ensembles_2021} in the sense that it puts weaker assumptions on the properties of the parameter space $\p$. We note that this is partially contained in \cite{Jerome_MCRF_2023}. In general, ensemble reachability as defined in Definition~\ref{def:reachbility}\,(iii) is an approximation property in the space $X_n(\p)$. For the space of continuous functions, approximation of continuous functions on compact subsets $\C$ is a classical topic. It is well-known in the theory of complex approximation, that we have to impose  topological conditions on the parameter space $\p$, in particular that it is has empty interior and its complement $\C \setminus \p$ is connected, cf. \cite[Ch.~2, \S~3,B]{gaier1987}.

	\medskip
	
	\begin{proposition}\label{prop:1}
		Let $\mathbf{P}$ be nonempty, compact, and  with empty interior such that $\C\setminus \p$ is connected. Then, a pair $(A,b)\in C_{n,n}(\p) \times C_{n}(\p)$ is uniformly ensemble reachable if the following conditions are satisfied:
		\begin{enumerate}
			\item[(a)] $( A ( \theta ) , b ( \theta ) )$ is reachable for all $\theta \in \mathbf{ P }$.
			\item[(b)] For all distinct parameters $\theta, \theta' \in \mathbf{ P }$,  the spectra  $\sigma \big( A(\theta)\big)$ and $\sigma \big( A(\theta') \big)$ are disjoint.
			\item[(c)] For each $\theta \in \mathbf{ P }$, the eigenvalues of $A(\theta )$ are simple.
		\end{enumerate}
	\end{proposition}
	
	\medskip

	The next lemma will be useful in the proof and may also be of independent interest.
	
	\begin{lemma}\label{lem:diagonal}
	Under the assumptions of Proposition~\ref{prop:1}, there is a continuously invertible  $T \in C_{n,n}(\p)$ such  that for every $\theta \in \p$ it holds
			\begin{align*}
		T(\theta)^{-1}A(\theta)T(\theta)=\begin{pmatrix}
			\lambda_1(\theta) & & \\
			& \ddots &\\
			& & \lambda_n(\theta)
		\end{pmatrix},\quad
		T(\theta)^{-1}b(\theta)= \begin{pmatrix}
			1\\
			\vdots\\
			1
		\end{pmatrix},
	\end{align*}
	where $\lambda_1,\dots,\lambda_n : \mathbf{P} \to \mathbb{C}$ are continuous and injective eigenvalue functions.
	\end{lemma}

\begin{proof}
	As all the eigenvalues of $A$ are simple, it follows from classical perturbation theory, cf.  \cite[Ch.~II~\S~5.2, Thm.~5.1]{Kato-1995}, that the eigenvalues depend continuously on $\theta$  and that there are continuous functions  $\lambda_1,\dots,\lambda_n : \mathbf{P} \to \mathbb{C}$  such that for each parameter $\theta \in\p$ we have
    $\sigma(A(\theta)) = \{ \lambda_1(\theta),\dots,\lambda_n(\theta)\}$. 
    
    Let $\lambda_i(\p)$ denote the image of $\lambda_i : \mathbf{P} \to \mathbb{C}$. As $\p$ is compact also $\lambda_i(\p)$ is compact.
    Assumptions~(b) and (c) imply 
    that $\lambda_i(\p)\cap\lambda_j(\p)=\emptyset$ for $i\neq j$.	 Together with assumption~(c), we conclude that the functions $\lambda_1,\dots,\lambda_n $ are injective.
	
	 For each $i \in \{1,...,n\}$, let $\Gamma_i$ be a contour enclosing $\lambda_i(\p)$ such that for all $j\neq i$ the images $\lambda_j(\p)$ are outside $\Gamma_i$. Define the associated Riesz projection
	\[
	P_i(\theta)
	=
	\frac{1}{2\pi i}
	\int_{\Gamma_i}
	(z I-A(\theta))^{-1}\, \di{z} .
	\]
	Then $P_i(\theta)$ depends continuously on $\theta$, has rank one, and projects onto the eigenspace corresponding to $\lambda_i(\theta)$, cf. \cite[Ch.~II~\S~5.1, Thm.~5.1]{Kato-1995}.
	
    For every $\theta\in \p$ we have $P_i(\theta)b(\theta)\neq 0$, $i=1,\ldots,n$, as otherwise the Kalman rank condition fails. More concretely, if, say, $P_1(\theta)b(\theta)= 0$, then using \cite[I.(5.21)]{Kato-1995} we have 
    for all $k$ that $P_1(\theta) A^k(\theta)b(\theta) = A^k(\theta)P_1(\theta) b(\theta) = 0$ and so the image of the Kalman matrix is contained in the lower dimensional subspace $\ker P_1(\theta)$, a contradiction.
	
	For $i=1,...,n$, we define
	\[
	v_i(\theta):= P_i(\theta)b(\theta) \neq 0.
	\]
	By construction, for every $i=1,...,n$ it holds that
	$
	A(\theta) v_i(\theta)=\lambda_i(\theta)v_i(\theta).
	$
Then, by setting
	\[
	S_1(\theta):=
	(v_1(\theta),\dots,v_n(\theta))
	\]
	we obtain a continuous diagonalization of $A(\theta)$. By the reachability assumtption it follows that every entry of $\tilde{b}(\theta) := S_1(\theta)^{-1}b(\theta)$ is nonzero.	
Defining the continuously invertible matrix $S_2(\theta) = \operatorname{diag}\bigl(\tilde b_1(\theta)^{-1},...,\tilde b_n(\theta)^{-1}  \bigr)$, a direct computation shows that the transformation $T(\theta):=S_1(\theta)S_2(\theta)$  yields the desired structure. This shows the assertion.
\end{proof}

	\begin{proof}[Proof of Proposition~\ref{prop:1}]
		Based on \cite[Cor.~3.1.2]{triggiani75}, it is shown in \cite[Thm.~1]{schonlein2016controllability} that the discrete-time and the continuous-time statements are equivalent.  Hence, it is sufficient to treat the discrete-time case.  Furthermore, using Lemma~\ref{lem:diagonal}, we may assume that
		\begin{align*}
			A(\theta)=\begin{pmatrix}
				\lambda_1(\theta) & & \\
				& \ddots &\\
				& & \lambda_n(\theta)
			\end{pmatrix},\quad
			b(\theta)= \begin{pmatrix}
				1\\
				\vdots\\
				1
			\end{pmatrix},
		\end{align*}
		where $\lambda_1,...,\lambda_n$ denote the continuous eigenvalue functions with pairwise disjoint images. It follows from \cite[Thm.~1]{schonlein2016controllability} that the pair $(A,b)$ is uniformly ensemble reachable, if  and only if  
	for every $f \in C_n( \p)$ and every $\varepsilon>0$ there is a polynomial $p \in \mathbb{C}[z]$ such that
		\begin{align}\label{eq:uer-polynomial}
		\max_{\theta \in\p} \| p(A(\theta ))b(\theta)	- f(\theta) \| < \e.
		\end{align}

		So, let $f \in C_n( \p)$ and $\eps>0$ be fixed. To prove the statement, we will construct a  polynomial $p \in \mathbb{C}[z]$ satisfying \eqref{eq:uer-polynomial}.  We make the ansatz
		\begin{equation}\label{eq:def-ansatz-polynomial}
			p(z) = \sum_{k=1}^n p_k(z)q_k(z),
		\end{equation}
		where the polynomials $p_1,...,p_n\in \C[z]$ and $q_1,...,q_n\in \C[z]$  are  constructed in two separate steps.
		
		In a first step, denoting $f= \begin{pmatrix}
		    f_1, \ldots, f_n
		\end{pmatrix}^\top$, we show that exist polynomials $p_1,...,p_n\in \C[z]$ such that
		\begin{equation}\label{eq:proof_prop_1}
			\max_{\theta \in \p}\| p_i(\lambda_i(\theta)) - f_i(\theta)\|< \tfrac{\eps}{3}  \quad \text{ for all } i =1,..,n.
		\end{equation}
		We note that, by conditions~(b) and~(c), the functions $\lambda_i \colon \p \to \lambda_i(\p)$ are homeomorphisms onto their respective image. Thus, condition~\eqref{eq:proof_prop_1} is equivalent to
		\begin{equation}
			\label{eq:fuer-mergelyan}
			\max_{z \in \lambda_i(\p)}\| p_i(z) - f_i(\lambda_i^{-1}(z))\|< \tfrac{\eps}{3}  \quad \text{ for all } i =1,..,n.
		\end{equation}
		To see that such polynomials exist, we will use Mergelyan's Theorem~\cite[Ch.~III, §~2, Sec.~A, Theorem~1]{gaier1987}. In order to apply it, we have to verify that sets $\lambda_1(\p),...,\lambda_n(\p)$ are compact with empty interior and that $\mathbb{C} \setminus \lambda_i(\p)$ is connected for all $i=1,...,n$. 	Obviously, the sets  $\lambda_1(\p),...,\lambda_1(\p)$ are compact with empty interior. Moreover, it follows from \cite[Corollary~2.2 and~2.5]{Born2022} that  the sets $\C\setminus \lambda_i(\p)$, $i=1,...,n$ are connected. Hence, we can apply Mergelyan's Theorem showing the existence of polynomials $p_1,\ldots,p_n\in \C[z]$ satisfying \eqref{eq:proof_prop_1}.

		Second we establish the existence of suitable polynomials $q_1,...,q_n\in \C[z]$ with the aim to approximate the indicator functions of $\lambda_i(\p)$. To this end, let
		$$\lambda(\p) = \bigcup_{i=1}^n \lambda_i(\p) $$
		and recall that the assumptions~(b) and~(c)  imply that the sets $\lambda_1(\p),..., \lambda_n(\p)$ are pairwise disjoint.
		To get the polynomials $q_1,...,q_n$, we will use Runge's Theorem~\cite[Ch.~III, §~1, Sec.~B, Theorem~2]{gaier1987}. Choose pairwise disjoint regions $U_i\subset \C$ that properly contain $\lambda_i(\p)$, $i=1,\ldots,n$, and set $U=\cup_{i=1}^n U_i$. Moreover, we choose holomorphic functions $h_1,...,h_n:U \to \mathbb{C}$ such that their restrictions to $\lambda(\p) $ satisfy
		\begin{align*}
			h_i|_{\lambda(\p) }\colon \lambda(\p) \to \C , \quad   h_i(z)=
			\begin{cases}
				1 & \text{ if } z \in \lambda_i (\p) \\
				0 & \text{ if } z \in \lambda(\p) \setminus \lambda_i(\p).
			\end{cases}
		\end{align*}
		By applying Runge's Theorem to $h_1,...,h_n$, there are polynomials $q_1,...,q_n$ such that
		\begin{align*}
			\sup_{z \in \lambda_i(\p)} | q_i(z) - h_i(z))| <  \frac{\varepsilon}{3 \, \sum_{j=1}^n \alpha_{i,j}},
		\end{align*}
		where
		$$\alpha_{i,j}:= \sup_{\theta \in \p}|p_i(\lambda_j(\theta))|.$$
		
		The final step is to see that the proposed polynomial $p \in \mathbb{C}[z]$ defined in \eqref{eq:def-ansatz-polynomial} satisfies \eqref{eq:uer-polynomial}. It holds that
		\begin{align*}
		\max_{\theta \in\p} \| p(A(\theta ))b(\theta)	- f(\theta) \|  \leq  \max_{\theta \in \p} \sup_{i=1,...,n}\left| \sum_{j=1}^n p_j(\lambda_i(\theta))q_j(a_i(\theta))-f_i(\theta)\right|.
		\end{align*}
		By construction, for each $i \in\{1,...,n\}$ and each $\theta \in\p$ it holds that
		\begin{multline*} \left|  \sum_{j=1}^n p_j(\lambda_i(\theta))q_j(\lambda_i(\theta))-f_i(\theta)\right| \\\leq
			| p_i(\lambda_i(\theta))q_i(\lambda_i(\theta))-f_i(\theta)|
			+ \left| \sum_{j=1,\, j\neq i}^n p_j(\lambda_i(\theta))q_j(\lambda_i(\theta))\right| \\
			\leq | p_i(\lambda_i(\theta))-f_i(\theta)| + |  p_i(\lambda_i(\theta)) | \cdot | q_i(\lambda_i(\theta)) - 1|   + \sum_{j=1,\, j\neq i}^n |p_j(\lambda_i(\theta))| \cdot |q_j(\lambda_i(\theta))| <\eps.
		\end{multline*}
		This shows the assertion.
	\end{proof}


\begin{thebibliography}{10}

\bibitem{agrachev-baryshnikov-sarychev2016}
A.~Agrachev, Y.~Baryshnikov, and A.~Sarychev.
\newblock Ensemble controllability by {Lie} algebraic methods.
\newblock {\em ESAIM: Control, Optimisation and Calculus of Variations},
  22(4):921--938, 2016.

\bibitem{agrachev2020control}
A.~Agrachev and A.~Sarychev.
\newblock Control in the spaces of ensembles of points.
\newblock {\em SIAM Journal on Control and Optimization}, 58(3):1579--1596,
  2020.

\bibitem{appell2004nonlinear}
J.~Appell, E.~De~Pascale, and A.~Vignoli.
\newblock {\em Nonlinear Spectral Theory}.
\newblock Walter de Gruyter, 2004.

\bibitem{Bar-Fer-Zab-2007-generic}
I.~Barag\~{a}na, V.~Fern\'{a}ndez, and I.~Zaballa.
\newblock Hermite indices and state feedback: generic case.
\newblock {\em Linear and Multilinear Algebra}, 55(2):113--120, 2007.

\bibitem{bcr2010}
K.~Beauchard, J.-M. Coron, and P.~Rouchon.
\newblock Controllability issues for continuous-spectrum systems and ensemble
  controllability of {B}loch equations.
\newblock {\em Communications in Mathematical Physics}, 296(2):525--557, 2010.

\bibitem{blondel1994simultaneous}
V.~Blondel.
\newblock {\em Simultaneous Stabilization of Linear Systems}, volume 191 of
  {\em Lecture Notes in Control and Information Sciences}.
\newblock Springer, Berlin, 1994.

\bibitem{Born2022}
M.~Born.
\newblock {\em On Lacunary Approximation of Mergelyan Type}.
\newblock doctoral thesis, Universit{\"a}t Trier, 2022.

\bibitem{bressan-LN-functional-analysis}
A.~Bressan.
\newblock {\em Lecture notes on Functional Analysis with Applications to Linear
  Partial Differential Equations}.
\newblock Number 143 in Graduate Texts in Mathematics. American Mathematical
  Society, Providence, RI, 2013.

\bibitem{chen2019mcss}
X.~{Chen}.
\newblock Structure theory for ensemble controllability, observability and
  duality.
\newblock {\em Math. Control Signals Syst.}, 31(1):1--40, 2019.

\bibitem{Xudong_SICON_2023}
X.~Chen.
\newblock Controllability issues of linear ensemble systems over
  multidimensional parameterization spaces.
\newblock {\em SIAM Journal on Control and Optimization}, 61(4):2425--2447,
  2023.

\bibitem{Chen-feedback-Banff}
X.~Chen.
\newblock Pole placement and feedback stabilization for discrete linear
  ensemble systems.
\newblock In M.~Belabbas, editor, {\em Geometry, Topology, and Control System
  Design: Proceedings of a Banff International Research Station Workshop},
  number~13 in AIMS on Applied Mathematics, pages 5--31. AIMS, 2025.

\bibitem{conway_functional_2nd}
J.~B. Conway.
\newblock {\em A Course in Functional Analysis}, volume~96 of {\em Graduate
  Texts in Mathematics}.
\newblock Springer-Verlag, New York, NY, second edition, 1990.

\bibitem{curtainzwart2020}
R.~F. Curtain and H.~Zwart.
\newblock {\em {Introduction to Infinite-Dimensional Systems Theory. A
  State-Space Approach}}, volume~71 of {\em Texts in Applied Mathematics}.
\newblock {Springer}, New York, 2nd edition, 2020.

\bibitem{Jerome_MCRF_2023}
B.~Danhane, J.~Loh{\'e}ac, and M.~Jungers.
\newblock Conditions for uniform ensemble output controllability, and
  obstruction to uniform ensemble controllability.
\newblock {\em Mathematical Control and Related Fields}, 14(3):1128--1175,
  2024.

\bibitem{JDE_ensembles_2021}
G.~Dirr and M.~Sch{\"o}nlein.
\newblock Uniform and ${L}^q$-ensemble reachability of parameter-dependent
  linear systems.
\newblock {\em Journal of Differential Equations}, 283:216--262, 2021.

\bibitem{banff_2023}
G.~Dirr and M.~Sch{\"o}nlein.
\newblock Ensemble controllability on various function spaces: Bounded and
  unbounded domains.
\newblock In M.~Belabbas, editor, {\em Geometry, Topology, and Control System
  Design: Proceedings of a Banff International Research Station Workshop},
  number~13 in AIMS on Applied Mathematics, pages 121--140. AIMS, 2025.

\bibitem{engelnagel}
K.-J. {Engel} and R.~{Nagel}.
\newblock {\em One-Parameter Semigroups for Linear Evolution Equations}.
\newblock Springer, Berlin, 2000.

\bibitem{fuhrmann2015mathematics}
P.~A. Fuhrmann and U.~Helmke.
\newblock {\em The Mathematics of Networks of Linear Systems}.
\newblock Springer International Publishing, 2015.

\bibitem{gaier1987}
D.~Gaier.
\newblock {\em Lectures on Complex Approximation}.
\newblock Birkh\"auser, Boston, MA, 1987.

\bibitem{nagelone}
G.~Greiner and R.~Nagel.
\newblock Spectral theory of semigroups on {B}anach spaces.
\newblock In R.~Nagel, editor, {\em One-Parameter Semigroups of Positive
  Operators}, volume 1184 of {\em Lecture Notes in Mathematics}, pages 60--97.
  Springer, 1986.

\bibitem{guth2023ensemble}
P.~A. Guth, K.~Kunisch, and S.~S. Rodrigues.
\newblock Ensemble feedback stabilization of linear systems.
\newblock {\em Applied Mathematics and Optimization}, 92:21, 2025.

\bibitem{hardt1996spectral}
V.~Hardt and E.~Wagenf{\"u}hrer.
\newblock Spectral properties of a multiplication operator.
\newblock {\em Mathematische Nachrichten}, 178(1):135--156, 1996.

\bibitem{hautussontag1986}
M.~{Hautus} and E.~D. {Sontag}.
\newblock New results on pole-shifting for parametrized families of systems.
\newblock {\em J. Pure Appl. Algebra}, 40:229--244, 1986.

\bibitem{helmke2014uniform}
U.~Helmke and M.~Sch{\"o}nlein.
\newblock Uniform ensemble controllability for one-parameter families of
  time-invariant linear systems.
\newblock {\em Systems \& Control Letters}, 71:69--77, 2014.

\bibitem{heymann1968comments}
M.~Heymann.
\newblock On pole assignment in multi-input controllable linear systems.
\newblock {\em IEEE Transactions on Automatic Control}, 13(6):748--749, 1968.

\bibitem{hinrichsen2026mathematicalII}
D.~Hinrichsen, A.~J. Pritchard, F.~Colonius, T.~Damm, A.~Ilchmann, B.~Jacob,
  and F.~R. Wirth.
\newblock {\em Mathematical Systems Theory II: Control, Observation,
  Realization, and Feedback}, volume~85 of {\em Texts in Applied Mathematics}.
\newblock Springer Nature, Cham, Switzerland, 2026.

\bibitem{hu1963homotopy}
S.-T. Hu.
\newblock {\em Homotopy Theory}, volume~8 of {\em Pure and Applied
  Mathematics}.
\newblock Academic Press, New York, NY, 1959.

\bibitem{kailath1980}
T.~Kailath.
\newblock {\em {Linear systems}}.
\newblock {Prentice-Hall, Inc., Englewood Cliffs Publ., N.J.}, 1980.

\bibitem{Kato-1995}
T.~Kato.
\newblock {\em Perturbation Theory for Linear Operators}.
\newblock Classics in Mathematics. Springer-Verlag, Berlin, 1995.

\bibitem{Kell75}
J.~L. Kelley.
\newblock {\em General Topology}.
\newblock Number~27 in Graduate Texts in Mathematics. Springer-Verlag, New
  York, NY, 1975.

\bibitem{LAZAR2022265}
M.~Lazar and J.~Lohéac.
\newblock Control of parameter dependent systems.
\newblock In E.~Tr\'{e}lat and E.~Zuazua, editors, {\em Numerical Control: Part
  A}, volume~23 of {\em Handbook of Numerical Analysis}, pages 265--306.
  Elsevier, 2022.

\bibitem{li2011}
J.-S. Li.
\newblock Ensemble control of finite-dimensional time-varying linear systems.
\newblock {\em IEEE Transactions on Automatic Control}, 56(2):345--357, 2011.

\bibitem{li2009ensemble}
J.-S. Li and N.~Khaneja.
\newblock Ensemble control of {B}loch equations.
\newblock {\em IEEE Transactions on Automatic Control}, 54(3):528--536, 2009.

\bibitem{li_tac_2016}
J.~S. Li and J.~Qi.
\newblock Ensemble control of time-invariant linear systems with linear
  parameter variation.
\newblock {\em IEEE Transactions on Automatic Control}, 61(10):2808--2820,
  2016.

\bibitem{li2020separating}
J.-S. Li, W.~Zhang, and L.~Tie.
\newblock On separating points for ensemble controllability.
\newblock {\em SIAM Journal on Control and Optimization}, 58(5):2740--2764,
  2020.

\bibitem{pritchard1981stability}
A.~J. Pritchard and J.~Zabczyk.
\newblock Stability and stabilizability of infinite-dimensional systems.
\newblock {\em {SIAM} Review}, 23(1):25--52, 1981.

\bibitem{ryan2014simultaneous}
E.~P. Ryan.
\newblock On simultaneous stabilization by feedback of finitely many
  oscillators.
\newblock {\em IEEE Transactions on Automatic Control}, 60(4):1110--1114, 2014.

\bibitem{schechter2002principles}
M.~Schechter.
\newblock {\em Principles of functional analysis}.
\newblock Number~36 in Graduate Texts in Mathematics. American Mathematical
  Society, Providence, RI, 2002.

\bibitem{laa_feedback}
M.~Sch{\"o}nlein.
\newblock Feedback equivalence and uniform ensemble reachability.
\newblock {\em Linear Algebra and its Applications}, 646:175–--194, 2021.

\bibitem{MCSS_Ensembles}
M.~Sch{\"o}nlein.
\newblock Polynomial methods to construct inputs for uniformly ensemble
  reachable linear systems.
\newblock {\em Mathematics of Control, Signals, and Systems}, 36(2):251--296,
  2024.

\bibitem{schonlein2016controllability}
M.~Sch{\"o}nlein and U.~Helmke.
\newblock Controllability of ensembles of linear dynamical systems.
\newblock {\em Mathematics and Computers in Simulation}, 125:3--14, 2016.

\bibitem{sontag_intro_families}
E.~D. Sontag.
\newblock An introduction to the stabilization problem for parametrized
  families of linear systems.
\newblock In {\em Linear algebra and its role in systems theory ({B}runswick,
  {M}aine, 1984)}, volume~47 of {\em Contemp. Math.}, pages 369--400. Amer.
  Math. Soc., Providence, RI, 1985.

\bibitem{sontag}
E.~D. {Sontag}.
\newblock {\em Mathematical Control Theory. Deterministic Finite Dimensional
  Systems}.
\newblock Springer, New York, NY, 2nd edition, 1998.

\bibitem{sontagwang1990}
E.~D. {Sontag} and Y.~{Wang}.
\newblock Pole shifting for families of linear systems depending on at most
  three parameters.
\newblock {\em Linear Algebra Appl.}, 137-138:3--38, 1990.

\bibitem{triggiani75}
R.~Triggiani.
\newblock Controllability and observability in {B}anach space with bounded
  operators.
\newblock {\em SIAM Journal on Control}, 13(2):462--491, 1975.

\bibitem{triggiani1975stabilizability}
R.~Triggiani.
\newblock On the stabilizability problem in {B}anach space.
\newblock {\em Journal of Mathematical Analysis and Applications},
  52(3):383--403, 1975.

\bibitem{triggiani1975pathological}
R.~Triggiani.
\newblock Pathological asymptotic behavior of control systems in {B}anach
  space.
\newblock {\em Journal of Mathematical Analysis and Applications},
  49(2):411--429, 1975.

\bibitem{Zeng_scl_2016}
S.~Zeng and F.~Allg{\"o}wer.
\newblock A moment-based approach to ensemble controllability of linear
  systems.
\newblock {\em Systems \& Control Letters}, 98:49--56, 2016.

\bibitem{shen2017discrete}
S.~Zeng, H.~{I}shii, and F.~{A}llg\"ower.
\newblock Sampled observability and state estimation of discrete ensembles.
\newblock {\em IEEE Trans. Autom. Contr.}, 62(5):2406--2418, 2017.

\bibitem{zeng2016ensemble}
S.~Zeng, S.~Waldherr, C.~Ebenbauer, and F.~Allg{\"o}wer.
\newblock Ensemble observability of linear systems.
\newblock {\em IEEE Transactions on Automatic Control}, 61(6):1452--1465, 2016.

\bibitem{zhang2018controllability}
W.~Zhang and J.-S. Li.
\newblock On controllability of time-varying linear population systems with
  parameters in unbounded sets.
\newblock {\em Systems \& Control Letters}, 118:94--100, 2018.

\bibitem{zuazua2014averaged}
E.~Zuazua.
\newblock Averaged control.
\newblock {\em Automatica}, 50(12):3077--3087, 2014.

\end{thebibliography}
\end{document}